\documentclass[11pt,reqno]{amsart}

\usepackage{amscd,amssymb,amsmath,amsthm}
\usepackage[arrow,matrix]{xy}
\usepackage{graphicx}
\usepackage{color}
\usepackage{cite}
\usepackage{tikz}

\newtheorem{thm}{Theorem}
\newtheorem{defn}{Definition}

\newtheorem{pro}{Proposition}
\newtheorem{rk}{Remark}

\newtheorem{ex}{Example}

\numberwithin{equation}{section} 

\newcommand{\bea}{\begin{eqnarray}}
\newcommand{\eea}{\end{eqnarray}}

\def\ann{\operatorname{ann}}

\begin{document}

\title[A class of nilpotent evolution algebras]{A class of nilpotent evolution algebras}

\author{B. A. Omirov,  U. A. Rozikov, M. V. Velasco}

 \address{B. \ A. \ Omirov\\ National University of Uzbekistan,
 4, University street, 100174, Tashkent, Uzbekistan.}
 \email {omirovb@mail.ru}

 \address{U.\ A.\ Rozikov\\ Institute of mathematics,
81, Mirzo Ulug'bek str., 100125, Tashkent, Uzbekistan.}
\email {rozikovu@yandex.ru}

\address{M.\ V. \ Velasco\\ Departamento de An\'{a}lisis Matem\'{a}tico
Facultad de Ciencias Universidad de Granada
18071- Granada, Spain.}
\email {vvelasco@ugr.es}

\begin{abstract} Recently, by A. Elduque and A. Labra a new technique and a type of an evolution
 algebra are introduced. Several nilpotent evolution algebras defined in terms of bilinear forms and
symmetric endomorphisms are constructed. The technique then used for the classification of the nilpotent evolution
algebras up to dimension five. In this paper we develop this technique for high dimensional evolution algebras.
We construct nilpotent evolution algebras of any type. Moreover, we show that, except the cases considered by
Elduque and Labra, this construction of nilpotent evolution algebras does not
give all possible nilpotent evolution algebras.
\end{abstract}
\maketitle



\section{Introduction} \label{sec:intro}

An evolution algebra over a field is an algebra with a basis on which
multiplication is defined by the product of distinct basis terms being zero and
the square of each basis element being a linear form in basis elements \cite{t}.

In study of any class of algebras, it is important to describe up
to isomorphism at least algebras of lower dimensions. In \cite{M} and \cite{Umlauf}, the
classifications of associative and nilpotent Lie algebras of low
dimensions were given.

About classifications of evolution algebras the following results are known:

In \cite{CLOR} (see also \cite{co}) two-dimensional evolution algebras over the complex numbers were classified.
For the classification of two-dimensional evolution algebras over the real numbers see \cite{Mu}.

Recently,  \cite{Cab} the authors classified three dimensional evolution algebras over a field
having characteristic different from 2 and in which there are
roots of orders 2, 3 and 7. It is proved that there are 116 types of
three-dimensional evolution algebras.

Very recently, in \cite{N} the authors studied the distribution
of finite-dimensional evolution algebras over any base field into isotopism\footnote{The concept of isotopism of algebras was introduced in \cite{A} as a generalization
of isomorphism. Two $n$-dimensional algebras $A$ and $B$ defined over
a field $K$ are isotopic if there exist three non-singular linear transformations $f, g$ and $h$
from $A$ to $B$ such that
$f(u)g(v) = h(uv)$, for all $u, v \in A$.} classes according
to their structure tuples and to the dimension of their annihilators. It is shown
 the existence of four isotopism classes of two-dimensional evolution algebras,
whatever the base field is. For the three-dimensional case
it is shown how to deal with the distribution into isotopism
classes of evolution algebras of higher dimensions.

In \cite{EL} (see also \cite{H}) a classification of indecomposable nilpotent
evolution algebras up to dimension five over algebraically closed fields of characteristic
not two is given. To do this in \cite{EL} the type and several invariant
subspaces related to the upper annihilating series of finite-dimensional
nilpotent evolution algebras are introduced.
A class of nilpotent evolution algebras, defined in terms of a nondegenerate,
symmetric, bilinear form and some commuting, symmetric, diagonalizable
endomorphisms relative to the form, are constructed.

In this paper we develop the methods of \cite{EL} for high dimensional evolution algebras.
We construct nilpotent evolution algebras of any type. We show that, except the cases considered by
Elduque and Labra, this construction of nilpotent evolution algebras does not
give all possible nilpotent evolution algebras.

\section{Basic definitions and facts}
{\it Evolution algebras.} Let $(E,\cdot)$ be an algebra over a field $K$. If it admits a basis
$\{e_1,e_2,\dots\}$, such that
\[
e_i \cdot e_j=
\begin{cases}
0, &\text{if \ $i\ne j$;}\\
\displaystyle \sum_{k}a_{ik}e_k, &\text{if \ $i=j$,}
\end{cases}
\]
then this algebra is called an {\it evolution algebra} \cite{t}. The basis is called a natural basis.  We denote by
$A=(a_{ij})$ the matrix of the structural constants of the evolution
algebra $E$.

It is known that an evolution algebra is commutative but not associative, in general.
For basic properties of the evolution algebra see \cite{t}.

For an evolution algebra $E$ and  $k \geq 1$  we introduce the following sequence
\begin{equation}\label{EE}
E^k=\sum_{i=1}^{k-1}E^iE^{k-i}.
\end{equation}
Since $E$ is a commutative algebra we obtain
\[E^k=\sum_{i=1}^{\lfloor k/2\rfloor}E^iE^{k-i},\]
where $\lfloor x\rfloor$ denotes the integer part of $x$.

\begin{defn} An evolution algebra $E$ is called  nilpotent if there
exists some $n\in \mathbb{N}$ such that $E^n=0$. The smallest $n$ such that
 $E^n=0$ is called the index of  nilpotency.
\end{defn}

The following theorem is known (see \cite{CLOR}).
\begin{thm} \label{27} An $n$-dimensional evolution algebra $E$ is nilpotent iff
the matrix of the structural constants corresponding to $E$ can be written as
\begin{equation}\label{5}
\widehat{A}=\left(
  \begin{array}{ccccc}
 0 & a_{12} & a_{13} &\dots &a_{1n} \\[1.5mm]
 0 & 0 & a_{23} &\dots &a_{2n} \\[1.5mm]
0 & 0 & 0 &\dots &a_{3n} \\[1.5mm]
\vdots & \vdots & \vdots &\cdots & \vdots \\[1.5mm]
0 & 0 & 0 &\cdots &0 \\
\end{array}
\right).
\end{equation}
\end{thm}

{\it Upper annihilating series.} Following \cite{EL} we introduce the following definitions:

Let $\mathcal E$ be an evolution algebra with a natural basis $B=\{e_1,\dots,e_n\}$ and matrix of structural constants
$A=(a_{ij})$. The {\it graph} $\Gamma(\mathcal E, B)=(V,E)$, with $V=\{1,\dots,n\}$ and $E=\{(i,j)\in V\times V: a_{ij}\ne 0\}$,
is called the graph attached to the evolution algebra $\mathcal E$ relative to the natural basis $B$.

\begin{defn} For an algebra $\mathcal A$ define the chain $\ann^i(\mathcal A)$, $i\geq 1$ by
$$\begin{array}{ll}
\ann^1(\mathcal A):=\ann(\mathcal A):=\{x\in\mathcal A: x\mathcal A=\mathcal A x=0\},\\[3mm]
\ann^i(\mathcal A)/\ann^{i-1}(\mathcal A):=\ann(\mathcal A/\ann^{i-1}(\mathcal A)).
\end{array}
$$
\end{defn}
 \begin{defn} The following series is called the upper annihilating series:
 $$0=\ann^0(\mathcal A)\subseteq \ann^1(\mathcal A)\subseteq\dots\subseteq \ann^r(\mathcal A)\subseteq\dots$$
 \end{defn}

 It is known that a non-associative  algebra (in particular an evolution algebra)
 is nilpotent if and only if its upper annihilating
 series riches $\mathcal A$, i.e., $\ann^r(\mathcal A)=\mathcal A$, for some $r\geq 1$.

 \begin{defn} Let $\mathcal A$ be a finite-dimensional nilpotent
 non-associative algebra over a field $\mathbb F$, and
 let $r$ be the lowest natural number such that
 $\ann^r(\mathcal A)=\mathcal A$.
 The type of $\mathcal A$ is the sequence
 $[n_1,n_2,\dots,n_r]$ such that $n_1+n_2+\dots+n_i=\dim_{\mathbb F}(\ann^i(\mathcal A))$, for all $i=1,2,\dots,r$.
 Thus
 $$n_i=\dim_{\mathbb F}(\ann^i(\mathcal A))-\dim_{\mathbb F}(\ann^{i-1}(\mathcal A)), \ \ i=1,2,\dots,r.$$
\end{defn}

\section{Nilpotent evolution algebras}

Consider a field $\mathbb F$ of characteristic not equal to 2. Let $\mathcal U$ be a vector space over $\mathbb F$
with $\dim_{\mathbb F} \mathcal U=n$.

\begin{defn}\label{dnk} Let $b : \mathcal U\times \mathcal U\to\mathbb F$ be a nondegenerate symmetric bilinear
form and let $f_i : \mathcal U\to \mathcal U$, $i=1,2,\dots,k-1$ be pairwise commuting, symmetric (relative to
$b$), diagonalizable endomorphisms. We define the algebra $E(\mathcal U, b, f_1,\dots, f_{k-1}):=
\mathcal U\times \underbrace{\mathbb F\times \dots\times \mathbb F}_k$ with multiplication
$$
\left(u,\alpha_1,\dots,\alpha_k\right)\left(v,\beta_1,\dots,\beta_k\right)=$$
\begin{equation}\label{nk}
\left(0, b(u,v), b(f_1(u),v)+\alpha_1\beta_1, b(f_2(u),v)+\alpha_2\beta_2, \dots, b(f_{k-1}(u),v)+\alpha_{k-1}\beta_{k-1}\right),
\end{equation}
for any $u,v\in \mathcal U$ and $\alpha_i, \beta_j\in \mathbb F$.
\end{defn}
\begin{pro}\label{p1} $E(\mathcal U, b, f_1, \dots, f_{k-1})$ is a nilpotent evolution algebra of type
$[\underbrace{1, 1,\dots 1}_k, n]$.
\end{pro}
\begin{proof} By assumptions there is an orthogonal basis $\{u_1,\dots, u_n\}$ of $\mathcal U$, relative
to $b$, consisting of common eigenvalues for $f_i$, $i=1,\dots,k-1$. Then
$$e_1=(u_1, 0, 0,\dots, 0), \dots, e_n=(u_n, 0, 0,\dots, 0),$$
$$e_{n+1}=(0, 1, 0, 0,\dots,0), \dots, e_{n+k}=(0, 0,\dots, 0, 1)$$ is a natural basis of $E=E(\mathcal U, b, f_1, \dots, f_{k-1})$.
Moreover, $e_ie_j=0$ if $i\ne j$ and
  \begin{equation}\label{mnk}
  e_i^2=\left\{\begin{array}{lllll}
  \sum_{j=1}^k\lambda_{ij}e_{n+j}, \ \ \mbox{if} \ \ i=1,\dots,n\\[2mm]
  e_{i+1}, \ \ \mbox{if} \ \ i=n+1,\dots,n+k-1\\[2mm]
  0, \ \ \mbox{if} \ \ i=n+k,
  \end{array}
  \right.
  \end{equation}
  where $$\lambda_{ij}=\left\{\begin{array}{ll}
  b(u_i,u_i), \ \ \mbox{if} \ \ j=1\\[2mm]
  b(f_{j-1}(u_i),u_i), \ \ \mbox{if} \ \ j=2,\dots, k.
  \end{array}
  \right.
  $$
  Now using multiplication (\ref{nk}) we calculate $\ann^j(E)$. We have
  $$\ann(E)=\{\left(u,\alpha_1,\dots,\alpha_k\right)\in E: \ \ \left(u,\alpha_1,\dots,\alpha_k\right)\left(v,\beta_1,\dots,\beta_k\right)=0, \forall \left(v,\beta_1,\dots,\beta_k\right)\in E\}$$ $$=\underbrace{0\times0\times\dots\times 0}_k\times \mathbb F.$$
  Note that the first zero in the RHS of this formula is $n$-dimensional zero-vector.
  $$\ann^2(E)=\{\left(u,\alpha_1,\dots,\alpha_k\right)\in E:$$ $$\left(u,\alpha_1,\dots,\alpha_k\right)\left(v,\beta_1,\dots,\beta_k\right)\in \ann(E), \forall \left(v,\beta_1,\dots,\beta_k\right)\in E\}$$ $$=\underbrace{0\times0\times\dots\times 0}_{k-1}\times\mathbb F\times \mathbb F.$$
  Using induction one can prove that
  $$\ann^j(E)=\underbrace{0\times0\times\dots\times 0}_{k-j+1}\times\underbrace{\mathbb F\times\dots\times \mathbb F}_j,\ \ j=1,\dots,k.$$
  and $\ann^{k+1}(E)=E$. Thus we have
   $$n_i=\dim_{\mathbb F}(\ann^i(E))-\dim_{\mathbb F}(\ann^{i-1}(E))=\left\{\begin{array}{ll}
   1, \ \ \mbox{if} \ \ i=1,\dots k\\[2mm]
   n, \ \ \mbox{if} \ \ i=k+1.
   \end{array}
   \right.
   $$
  \end{proof}
The following proposition shows that Proposition \ref{p1} does not give all nilpotent evolution algebras of type $[\underbrace{1, \dots, 1}_k, n]$.
\begin{pro} \label{p2} For each $k\geq 4$ and $n\geq 1$ there is an evolution algebra $E$ of type $[\underbrace{1, \dots, 1}_k, n]$,
which is non-isomorphic to $E'=E(\mathcal U, b, f_1,\dots,f_{k-1})$ for any collection   $(\mathcal U, b, f_1,\dots,f_{k-1})$
as in Definition \ref{dnk}.
\end{pro}
\begin{proof} Let $E$ be of type $[\underbrace{1, \dots, 1}_k, n]$ and $\{h_1,\dots,h_n, h_{n+1},\dots,h_{n+k}\}$ be
 a natural basis of this algebra. Moreover,
$$\ann^j(E)=\underbrace{0\times0\times\dots\times 0}_{k-j+1}\times\underbrace{\mathbb F\times\dots\times \mathbb F}_j,\ \ j=1,\dots,k.$$
  and $\ann^{k+1}(E)=E$. Then we choose $E$ such that $h_ih_j=0$ if $i\ne j$ and
  \begin{equation}\label{bnk}
  h_i^2=\left\{\begin{array}{lllllll}
  h_{n+1}, \ \ \mbox{if} \ \ i=1,\dots,n\\[2mm]
  h_{i+1}+h_{i+2}, \ \ \mbox{if} \ \ i=n+1, n+2, \dots, n+k-2\\[2mm]
  h_{n+k}, \ \ \mbox{if} \ \ i=n+k-1\\[2mm]
  0, \ \ \mbox{if} \ \ i=n+k.
  \end{array}
  \right.
  \end{equation}

 We shall show that there is no a change from basis $\{h_i\}$ (of algebra $E$) with multiplication (\ref{bnk}) to the
 basis $\{e_i\}$ (of algebra denoted by $E'$) with multiplication (\ref{mnk}).
 We note that if such $\varphi$, (where $\det\varphi\ne 0$) exists then
 $$\varphi(\ann^m(E))=\ann^m(E'), \ \ 1\leq m\leq k+1.$$
 Moreover, by Corollary 3.6 of \cite{EL}, $\varphi$ has the following block structure:
 $$\left(\begin{array}{ccccc}
 *&0&0&\dots&*\\
 0&*&0&\dots&*\\
 0&0&*&\dots&*\\
 \vdots&\vdots&\vdots&\ddots&\vdots\\
 0&0&0&\dots&*,
 \end{array}
 \right)$$
 i.e.
 $$\varphi(h_i)=\gamma_{i}e_i+\delta_{i}e_{n+k}, \ \ \gamma_i, \delta_i\in \mathbb F.$$
 Thus $\det(\varphi)=\gamma_1\dots\gamma_{n+k}$.
 From (\ref{bnk}) we get
 $$h_{i}^2h_{i+1}^2=h_{i+2}^2, \ \ i=n+1, n+2, \dots, n+k-2.$$
 Consequently,
 $$\varphi(h_{i+2}^2)=\varphi(h_{i}^2)\varphi(h_{i+1}^2)=\varphi(h_i)^2\varphi(h_{i+1})^2=(\gamma_ie_{i}+\delta_ie_{n+k})^2
 (\gamma_{i+1}e_{i+1}+\delta_{i+1}e_{n+k})^2$$
 $$ =\gamma_i^2e_i^2\gamma_{i+1}^2e_{i+1}^2=\gamma_i^2\gamma_{i+1}^2e_{i+1}e_{i+2}=0$$
 for each $i=n+1, n+2, \dots, n+k-2.$
 Using this we obtain
 $$0=\varphi(h_{i+2}^2)=\varphi(h_{i+2})^2=\gamma^2_{i+2}e_{i+2}^2=\gamma^2_{i+2}e_{i+3},$$
 hence $\gamma_{i+2}=0$ for each $i=n+1,\dots,n+k-3$.
 Thus if $k\geq 4$ then $\det(\varphi)=0$, i.e. there is no isomorphism between $E$ and $E'$.
 \end{proof}
 \begin{rk} In \cite{EL} for $k\leq 3$ it is shown that any algebra of type $[\underbrace{1, \dots, 1}_k, n]$ is isomorphic
 to $E(\mathcal U, b, f_1,\dots,f_{k-1})$ for some $(\mathcal U, b, f_1,\dots,f_{k-1})$. Proposition \ref{p2} shows that this kind of result  is not true for any $k>3$.
 \end{rk}
\begin{defn}\label{d1} Let $b:\mathcal U\times \mathcal U\to \mathbb F$ be a nondegenerate, symmetric, bilinear form and let
$0\ne u\in \mathcal U$. Define the algebra $$E_{lr}(\mathcal U,b,u):= \underbrace{\mathbb F\times \mathbb F\times \dots \times \mathbb F}_l\times \mathcal U\times  \underbrace{\mathbb F\times \mathbb F\times \dots \times \mathbb F}_r$$ with multiplication
$$(\alpha_1,\dots,\alpha_l,x,\beta_1,\dots,\beta_r)(\gamma_1,\dots,\gamma_l,y,\delta_1,\dots,\delta_r)=$$
\begin{equation}\label{mul}(0,\alpha_1\gamma_1,\alpha_2\gamma_2,\dots, \alpha_{l-1}\gamma_{l-1}, \alpha_l\gamma_lu, b(x,y), \beta_1\delta_1, \dots, \beta_{r-1}\delta_{r-1}),
\end{equation}
for any $x,y\in \mathcal U$ and $\alpha_i, \beta_j,\gamma_k,\delta_m\in \mathbb F$.
\end{defn}
\begin{thm} $E_{lr}(\mathcal U,b,u)$ is a nilpotent evolution algebra of type $[\underbrace{1,1,\dots,1}_l,n,\underbrace{1,1,\dots,1}_r]$.
\end{thm}
\begin{proof} Let $\{u_1,\dots,u_n\}$ be an orthogonal basis of $\mathcal U$ related to $b$. Then using (\ref{mul}) it is easy to see that
 $$e_1=(1,0,\dots,0), e_2=(0,1,0,\dots 0), \dots ,e_l=(\underbrace{0,\dots,0,1}_l, 0,\dots,0),$$
  $$e_{l+1}=(\underbrace{0,\dots,0}_l,u_1,\underbrace{0,\dots,0}_r), \dots ,e_{l+n}=(\underbrace{0,\dots,0}_l,u_n,\underbrace{0,\dots,0}_r),$$
  $$ e_{l+n+1}=(0,\dots,0,\underbrace{1,0,\dots,0}_r), \dots, e_{l+n+r}=(0,\dots,0,\underbrace{0,0,\dots,1}_r)$$
  is a natural basis of $E=E_{lr}(\mathcal U,b,u)$ making it an evolution algebra.
  Moreover, $e_ie_j=0$ if $i\ne j$. If $u=\sum_{k=1}^nc_ku_k$ and $\lambda_k=b(u_k,u_k)$ then
  \begin{equation}\label{ma}
  e_i^2=\left\{\begin{array}{lllll}
  e_{i+1}, \ \ \mbox{if} \ \ i=1,\dots,l-1\ \ \mbox{and} \ \ i=l+n+1,l+n+2,\dots, l+n+r-1\\[2mm]
  \sum_{k=1}^nc_ke_{l+k},\ \ \mbox{if} \ \ i=l\\[2mm]
  \lambda_ie_{l+n+1}, \ \ \mbox{if} \ \ i=l+j, \ \ j=1,\dots,n\\[2mm]
  0, \ \ \mbox{if} \ \ i=l+n+r.
  \end{array}
  \right.
  \end{equation}
  In Figure \ref{f1s} the graph of the evolution algebra with multiplication (\ref{ma}) is given.
  \begin{center}
  \begin{figure}
\begin{tikzpicture}  [-latex , scale=.8,auto=left, every node/.style={circle,top color=white!200}]
  \node (n1) at (-10,0) {$e_1$};
  \node (n2) at (-8,0)  {$e_2$};
  \node (n3) at (-6,0)  {$\dots$};
  \node (n4) at (-4,0) {$e_l$};
  \node (n5) at (-2,0)  {$\vdots$};
  \node (n6) at (-2,4)  {$e_{l+1}$};
  \node (n7) at (-2,-4)  {$e_{l+n}$};
  \node (n8) at (0,0) {$e_{l+n+1}$};
  \node (n9) at (2,0) {$e_{l+n+2}$};
  \node (n10) at (4,0) {$\dots$};
  \node (n11) at (6,0) {$e_{l+n+r}$};
  \node (n12) at (-2,2)  {$e_{l+2}$};
  \node (n13) at (-2,-2)  {$e_{l+n-1}$};
  \foreach \from/\to in {n1/n2,n2/n3,n3/n4,n4/n5,n4/n6,n4/n7,n4/n12,n4/n13,n5/n8,n6/n8,n7/n8,n12/n8,n13/n8,n8/n9,n9/n10,n10/n11}
    \draw (\from) -- (\to);
\end{tikzpicture}
\caption{The graph of the evolution algebra given by multiplication (\ref{ma}).} \label{f1s}
\end{figure}
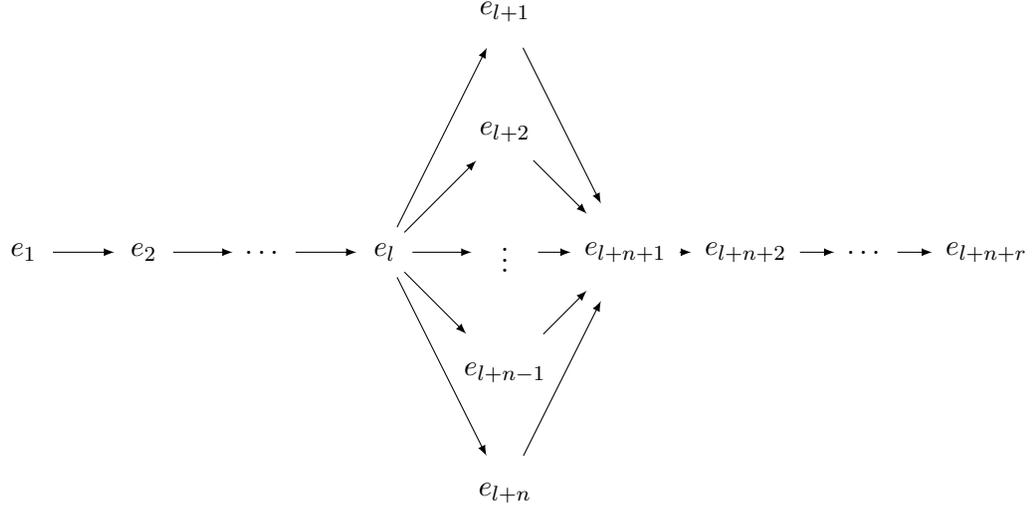
\end{center}

   Now we check the type of this evolution algebra.
 From (\ref{mul}) it follows that
 $$\ann(E)=\underbrace{0\times0\times\dots \times 0}_{l+r}\times \mathbb F;$$
 Note that in the last formula and below one $0$ is $n$-dimensional zero-vector.
 $$ \ann^2(E)=\left\{(\alpha_1,\dots,\alpha_l,x,\beta_1,\dots,\beta_r)\in E:\right. $$ $$ (\alpha_1,\dots,\alpha_l,x,\beta_1,\dots,\beta_r)(\gamma_1,\dots,\gamma_l,y,\delta_1,\dots,\delta_r)=$$
$$(0,\alpha_1\gamma_1,\alpha_2\gamma_2,\dots, \alpha_{l-1}\gamma_{l-1}, \alpha_l\gamma_lu, b(x,y), \beta_1\delta_1, \dots, \beta_{r-1}\delta_{r-1})\in \underbrace{0\times0\times\dots \times 0}_{l+r}\times \mathbb F,$$
  $$\ \ \mbox{for any} \ \ \left.(\gamma_1,\dots,\gamma_l,y,\delta_1,\dots,\delta_r)\in E\right\}=\underbrace{0\times0\times\dots \times 0}_{l+r-1}\times \mathbb F\times \mathbb F;$$
 Similarly by induction one shows that
 $$\ann^{j}(E)=\underbrace{0\times0\times\dots \times 0}_{l+r-j+1}\times \underbrace{\mathbb F\times \dots \times\mathbb F}_j, \ \ j=1,\dots,r.$$
$$ \ann^{r+1}(E)=\left\{(\alpha_1,\dots,\alpha_l,x,\beta_1,\dots,\beta_r)\in E: \right.$$ $$ (\alpha_1,\dots,\alpha_l,x,\beta_1,\dots,\beta_r)(\gamma_1,\dots,\gamma_l,y,\delta_1,\dots,\delta_r)$$
$$=(0,\alpha_1\gamma_1,\alpha_2\gamma_2,\dots, \alpha_{l-1}\gamma_{l-1}, \alpha_l\gamma_lu, b(x,y), \beta_1\delta_1, \dots, \beta_{r-1}\delta_{r-1})\in$$ $$ \underbrace{0\times0\times\dots \times 0}_{l+1}\times \underbrace{\mathbb F\times \dots \times\mathbb F}_r,\ \ \mbox{for any} \ \ \left.(\gamma_1,\dots,\gamma_l,y,\delta_1,\dots,\delta_r)\in E\right\}$$ $$=\underbrace{0\times0\times\dots \times 0}_{l}\times\mathcal U\times \underbrace{\mathbb F\times \dots \times\mathbb F}_r;$$
  Now using induction over $m$ one can show that
   $$\ann^{r+1+m}(E)=\underbrace{0\times0\times\dots \times 0}_{l-m}\times\underbrace{\mathbb F\times \dots \times\mathbb F}_m\times \mathcal U\times \underbrace{\mathbb F\times \dots \times\mathbb F}_r , \ \ m=1,\dots,l.$$
   Thus we have
   $$n_i=\dim_{\mathbb F}(\ann^i(E))-\dim_{\mathbb F}(\ann^{i-1}(E))=\left\{\begin{array}{ll}
   1, \ \ \mbox{if} \ \ i\ne r+1\\[2mm]
   n, \ \ \mbox{if} \ \ i=r+1.
   \end{array}
   \right.
   $$
\end{proof}

\begin{thm}\label{t3a}  If $l \geq 2$ or $r\geq 2$ then for each $n\geq 1$ there is an evolution algebra $E$ of type $[\underbrace{1,1,\dots,1}_l,n,\underbrace{1,1,\dots,1}_r]$ which is not isomorphic to an algebra $E_{lr}(\mathcal U, b, u)$, for any $(\mathcal U, b, u)$ as in Definition \ref{d1}.
\end{thm}
\begin{proof} Let $B=\{h_1,\dots,h_{l+n+r}\}$ be a natural basis of $E$ with type $[\underbrace{1,1,\dots,1}_l,n,\underbrace{1,1,\dots,1}_r]$. Then
$$ ann^j(E)=\left\{\begin{array}{llll}
\bigoplus_{i=1}^j\mathbb Fe_{l+n+r-i+1}, \ \ \mbox{if} \ \ j=1,\dots,r\\[2mm]
\bigoplus_{i=1}^n\mathbb Fe_{l+i}\bigoplus_{s=1}^r\mathbb Fe_{l+n+s}, \ \ \mbox{if} \ \ j=r+1\\[2mm]
\bigoplus_{t=1}^m\mathbb Fe_{l-t+1}\bigoplus_{i=1}^n\mathbb Fe_{l+i}\bigoplus_{s=1}^r\mathbb Fe_{l+n+s}, \ \ \mbox{if} \ \ j=r+1+m, \, m=1,\dots,l.
\end{array}
\right.
$$
Using these formulas we get the following multiplication table:
 \begin{equation}\label{ma1}
  \begin{array}{lllll}
  h_i^2=\sum_{j=i+1}^{l+n+r}\alpha_{i,j}h_j, \ \ \mbox{if} \ \ i=1,\dots,l\\[3mm]
  h_{l+t}^2=\sum_{j=l+n+1}^{l+n+r}\alpha_{l+t,j}h_j, \ \ \mbox{if} \ \ t=1,\dots,n\\[3mm]
  h_{l+n+m}^2=\sum_{j=l+n+m+1}^{l+n+r}\alpha_{l+n+m,j}h_j, \ \ \mbox{if} \ \ m=1,\dots,r-1\\[3mm]
  h_{l+n+r}^2=0.
  \end{array}
    \end{equation}
    Thus the matrix of structural constants of this algebra has upper triangular form with zeros in the diagonal (see Theorem \ref{27}).

   {\it Case:} $r\geq 2$.  Here we consider a particular case of (\ref{ma1}), i.e., we take
 \begin{equation}\label{ma2}
  \begin{array}{lllll}
 h_i^2=\sum_{j=i+1}^{l+n+r}\alpha_{i,j}h_j, \ \ \mbox{if} \ \ i=1,\dots,l\\[3mm]
  h_{l+1}^2=ch_{l+n+r},\, c\ne 0 \\[3mm]
  h_{l+t}^2=\sum_{j=l+n+1}^{l+n+r}\alpha_{l+t,j}h_j, \ \ \mbox{if} \ \ t=2,\dots,n\\[3mm]
  h_{l+n+m}^2=\sum_{j=l+n+m+1}^{l+n+r}\alpha_{l+n+m,j}h_j, \ \ \mbox{if} \ \ m=1,\dots,r-1\\[3mm]
  h_{l+n+r}^2=0.
  \end{array}
    \end{equation}
    We assume that there exists a change $\psi$ of basis $\{h_1, \dots, h_{l+n+r}\}$, of algebra $E$,
 with multiplication (\ref{ma2}) to the basis $\{e_1, \dots, e_{l+n+r}\}$, of algebra $E'$,
 with multiplication (\ref{ma}).
 Since $\psi(\ann^j(E))=\ann^j(E')$, we get for $\psi$ the following equalities:
 $$\psi(h_i)=\left\{\begin{array}{lllll}
 \sum_{j=i}^{l+n+r}\gamma_{ij}e_j, \ \ i=1,\dots, l;\\[3mm]
 \sum_{j=l+1}^{l+n+r}\gamma_{ij}e_j, \ \ i=l+1,\dots, l+n;\\[3mm]
\sum_{j=i}^{l+n+r}\gamma_{ij}e_j, \ \ i=l+n+1,\dots, l+n+r.
\end{array}
\right.$$
 By these formulas we have
 $$\det\psi=\det(\gamma_{ij})_{i,j=1}^l\cdot\det(\gamma_{ij})_{i,j=l+1}^{l+n}\cdot\det(\gamma_{ij})_{i,j=l+n+1}^{l+n+r}.$$
 For any $p\in \{l+1,\dots,l+n\}$, we have
 $$\psi(h_ph_{l+n+r-1})=\psi(h_p)\psi(h_{l+n+r-1})=\sum_{j=l+1}^{l+n+r}\gamma_{pj}\gamma_{l+n+r-1,j}e^2_j$$ $$
 =\sum_{j=l+1}^{l+n+r-2}\gamma_{pj}\gamma_{l+n+r-1,j}e^2_j+\gamma_{p,l+n+r-1}\gamma_{l+n+r-1,l+n+r-1}e_{l+n+r}=0.$$
 In particular, from the last equality we get the following system of equations
 $$\gamma_{p,l+n+r-1}\gamma_{l+n+r-1,l+n+r-1}=0, \ \ \mbox{for all} \ \ p\in  \{l+1,\dots,l+n\}.$$
We note that $\gamma_{l+n+r-1,l+n+r-1}\ne 0$, consequently
\begin{equation}\label{pl}
\gamma_{p,l+n+r-1}=0, \ \ \mbox{for all}  \ \ p\in \{l+1, \dots, l+n\}.
\end{equation}
 Now consider
 \begin{equation}\label{l+1}
\psi(h_{l+1}^2)=\psi(ch_{l+n+r})=c\psi(h_{l+n+r})=c\gamma_{l+n+r,l+n+r}e_{l+n+r}.
\end{equation}
 Since $\det\psi\ne 0$ we have $\gamma_{l+n+r,l+n+r}\ne 0$.
 On the other hand we have
  \begin{equation}\label{l+2}
\psi(h_{l+1}^2)=\psi(h_{l+1})^2=\left(\sum_{j=l+1}^{l+n+r}\gamma_{l+1,j}e_j\right)^2 =
\sum_{j=l+1}^{l+n+r-2}\gamma^2_{l+1,j}e^2_j+\gamma_{l+1,l+n+r-1}e_{l+n+r}.
\end{equation}
By (\ref{pl})-(\ref{l+2}) we get the following contradiction
$$0\ne c\gamma_{l+n+r,l+n+r}=\gamma_{l+1,l+n+r-1}=0.$$
Thus such $\psi$ does not exist.\\[4mm]

 {\it Case:} $l\geq 2$. Here we consider the following particular case of (\ref{ma1}):
 \begin{equation}\label{ma12}
  \begin{array}{lllll}
  h_i^2=h_{i+1}+h_{i+2}, \ \ \mbox{if} \ \ i=1,\dots,l\\[3mm]
  h_{l+t}^2=\sum_{j=l+n+1}^{l+n+r}\alpha_{l+t,j}h_j, \ \ \mbox{if} \ \ t=1,\dots,n\\[3mm]
  h_{l+n+m}^2=h_{l+n+m+1}+h_{l+n+m+2}, \ \ \mbox{if} \ \ m=1,\dots,r-2\\[3mm]
  h_{l+n+r-1}^2=h_{l+n+r}\\[3mm]
  h_{l+n+r}^2=0.
  \end{array}
    \end{equation}
   Assume that there exists a change $\varphi$ of basis $\{e_1, \dots, e_{l+n+r}\}$, of algebra $E'$,
 with multiplication (\ref{ma}) to the basis $\{h_1, \dots, h_{l+n+r}\}$, of algebra $E$,
 with multiplication (\ref{ma2}). Then since
 $$\varphi(\ann^j(E'))=\ann^j(E), \ \ j\geq 1,$$
 we get
  $$\varphi(e_i)=\sum_{j=i}^{l+n+r}\gamma_{ij}h_j, \ \ i=1,\dots, l.$$
 To have $\det(\varphi)\ne 0$ it is necessary that $\gamma_{11}\gamma_{22}\dots\gamma_{l+n+r,l+n+r}\ne 0$.

  From (\ref{ma}) we get
 $$e_{i}^2e_{i+1}^2=e_{i+1}e_{i+2}=0, \ \ i=1, 2, \dots, l-2.$$
 $$e_{l-1}^2e_{l}^2=e_{l}\left(\sum_{k=1}^nc_ke_{l+k}\right)=0.$$
 $$e_{l}^2e_{l+1}^2=\left(\sum_{k=1}^nc_ke_{l+k}\right)\lambda_1 e_{l+n+1}=0.$$
 $$e_{l+1}^2e_{l+2}^2=\lambda_1\lambda_2 e^2_{l+n+1}=\lambda_1\lambda_2e_{l+n+2}\ne 0.$$
  Consequently,
  \begin{equation}\label{o1}
  \varphi(e_{i}^2e_{i+1}^2)=0, \ \ \mbox{for each} \ \ i=1, 2, \dots, l.
  \end{equation}
 On the other hand we have
 $$
  \varphi(e_{i}^2e_{i+1}^2)=\left( \sum_{j=i}^{l+n+r}\gamma_{ij}h_j \right)^2
  \left( \sum_{j=i+1}^{l+n+r}\gamma_{i+1,j}h_j \right)^2=$$
  $$
  \left( \gamma^2_{ii}h_i^2+\dots \right)
  \left( \gamma_{i+1,i+1}^2h_{i+1}^2+\dots \right)=$$
  $$\left( \gamma^2_{ii}(h_{i+1}+h_{i+2})+\dots \right)
  \left( \gamma_{i+1,i+1}^2(h_{i+2}+h_{i+3})+\dots \right)=$$
  \begin{equation}\label{o2}
  \gamma^2_{ii}\gamma_{i+1,i+1}^2h^2_{i+2}+\dots , \ \ \mbox{for each} \ \ i=1, 2, \dots, l-1.
  \end{equation}
 By (\ref{o1}) from the last equality we get $\gamma_{ii}\gamma_{i+1,i+1}=0$ for each $i=1,\dots,l-1$.
 Thus for $l\geq 2$ we see that $\det(\varphi)=0$.
\end{proof}
 \begin{rk} In \cite{EL} it was shown that if $E$ is a nilpotent evolution algebra of type $[1,n,1]$
  then $E$ is isomorphic to an algebra  $E_{11}(\mathcal U,b,u)$, (i.e., $l=r=1$), for some $(U,b,u)$ as
  in Definition \ref{d1}. Theorem \ref{t3a} shows that this kind of result is not true for any $l,r$ when at least one of them $>1$.
 \end{rk}

\begin{defn}\label{d2} Let $b:\mathcal U\times \mathcal U\to \mathbb F$ be a non degenerate, symmetric, bilinear form.
Define the algebra $$E(\mathcal U,b):=\mathcal U\times  \mathbb F\times \mathbb F$$ with multiplication
$$(u,\alpha_1,\alpha_2)(v,\beta_1,\beta_2)=(0,0,b(u,v))$$
\end{defn}
\begin{pro} $E(\mathcal U,b)$ is a nilpotent evolution algebra of type $[2,n]$.
\end{pro}
\begin{proof} Let $\{u_1,\dots,u_n\}$ be an orthogonal basis of $\mathcal U$ related to $b$. Then $\{e_i=(u_i,0,0),i=1,\dots,n; e_{n+1}=(0,1,0), e_{n+2}=(0,0,1)\}$ is a natural basis of $E=E(\mathcal U,b)$. Moreover $e_ie_j=0$, $i\ne j$ and
$$e_i^2=\lambda_i(0,0,1)=\lambda_ie_{n+2}, \ \ i=1,\dots,n, \ \ \mbox{with} \ \ \lambda_i=b(u_i,u_i),$$
$$e_{n+1}^2=e_{n+2}^2=0.$$
It is easy to see that
$$\ann(E)=0\times \mathbb F\times\mathbb F, \ \ \ann^2(E)=E.$$
Thus $n_1=2$, $n_2=n$.
\end{proof}

\subsection{A construction of a nilpotent algebra of type $[n_1,n_2,\dots,n_k]$}
Consider a field $\mathbb F$ of characteristic not 2. Let $\mathcal U_i$ be a vector space over $\mathbb F$
with $\dim_{\mathbb F} \mathcal U_i=m_i$, $i=1,\dots,k$. Denote by $\{u_{i1}, \dots, u_{im_i}\}$ the basis elements of
$\mathcal U_i$, $i=1,\dots,k$.

\begin{defn}\label{dnn} Let $\xi_i : \mathcal U_i\otimes \mathcal U_i\to\mathcal U_{i+1}$, $i=1,\dots,k-1$
be symmetric bilinear mappings, such that $\xi_i(u_{ip}, u_{iq})=0, \, p\ne q$. We define the algebra
$$\mathbb E=\mathcal U_1\times \mathcal U_2\times\dots\times \mathcal U_k$$ with multiplication
$$
\left(x_1,x_2,\dots,x_k\right)\left(y_1,y_2,\dots,y_k\right)=$$
\begin{equation}\label{nkk}
\left(0, \xi_1(x_1,y_1), \xi_2(x_2,y_2), \dots, \xi_{k-1}(x_{k-1},y_{k-1})\right),
\end{equation}
for any $x_i,y_i\in \mathcal U_i$.
\end{defn}
\begin{thm}\label{t3} The algebra $\mathbb E=\mathcal U_1\times \mathcal U_2\times\dots\times \mathcal U_k$ is a nilpotent evolution algebra of type
$[n_1,n_2, \dots, n_k]$, with $n_i=m_{k-i+1}$.
\end{thm}
\begin{proof} We note that the following is a natural basis of $\mathbb E$:
$$\{e_{11},e_{12},\dots, e_{1m_1};
 e_{21},e_{22},\dots, e_{2m_2};\dots ; e_{k1},e_{k2},\dots, e_{km_k}\},$$
 where
 $$e_{ij}=(0,\dots,0, u_{ij}, 0,\dots,0)\in \mathbb E, \ \ i=1,\dots,k, \, j=1,\dots, m_i.$$
 Moreover we have $e_{ij}e_{pq}=0$, $(i,j)\ne (p,q)$ and
 $$e_{ij}^2=(0,\dots,0,\xi_i(u_{ij},u_{ij}),0,\dots,0), \ \ i=1,\dots,k-1, j=1,\dots,m_i,$$
 $$e_{kj}^2=0, \ \ j=1,\dots, m_k$$

 It is easy to see that
 $$\ann^1(\mathbb E)=0\times 0\times \dots 0\times \mathcal U_k, \ \ \mbox{i.e.}, \ \
 \dim \left(\ann(\mathbb E)\right)=m_k.$$
 $$\ann^2(\mathbb E)=0\times 0\times \dots 0 \times\mathcal U_{k-1}\times \mathcal U_k,,
 \ \ \mbox{i.e.}, \ \ \dim \left(\ann^2(\mathbb E)\right)=\sum_{s=k-1}^km_s.$$
 $$\dots\dots\dots$$
  $$\ann^j(\mathbb E)=0\times 0\times \dots 0\times \mathcal U_{k-j+1}\times\dots\times\mathcal U_{k-1}\times \mathcal U_k,, \ \ \mbox{i.e.}, \ \ \dim \left(\ann^j(\mathbb E)\right)=\sum_{s=k-j+1}^km_s.$$
  Thus we have
  $$n_i=\sum_{s=k-i+1}^km_s-\sum_{s=k-i+2}^km_s=m_{k-i+1}.$$
  Then corresponding algebra is of type $[n_1,n_2,\dots,n_k]$.
\end{proof}
\begin{ex} Let $\mathcal V_1$ be an evolution algebra with multiplication $\xi_1:\mathcal V_1\otimes \mathcal V_1\to \mathcal V_1$,
and $\mathcal V_2:=\xi_1(\mathcal V_1,\mathcal V_1)$. Assume a multiplication $\xi_2$ is given on $\mathcal V_2$ such that $(\mathcal V_2,\xi_2)$ is an evolution algebra and $\mathcal V_3:=\xi_2(\mathcal V_2,\mathcal V_2)$.  Consequently, define an evolution algebra $(\mathcal V_{i-1},\xi_{i-1})$ with $\xi_{i-1}(\mathcal V_{i-1},  \mathcal V_{i-1})=\mathcal V_i$, \, $i=1,\dots,k$.
Consider $$\mathcal E=\mathcal V_1\times \mathcal V_2\times\dots\times \mathcal V_k$$ with multiplication
$$
\left(x_1,x_2,\dots,x_k\right)\left(y_1,y_2,\dots,y_k\right)=$$
\begin{equation}\label{nkv}
\left(0, \xi_1(x_1,y_1), \xi_2(x_2,y_2), \dots, \xi_{k-1}(x_{k-1},y_{k-1})\right),
\end{equation}
for any $x_i,y_i\in \mathcal V_i$. Then by Theorem \ref{t3} the algebra $\mathcal E$ is a nilpotent evolution algebra of type
$[n_1,n_2, \dots, n_k]$, with $n_i=\dim\left(\mathcal V_{k-i+1}\right)$.

\end{ex}

\section*{ Acknowledgements}

The work partially supported by Projects MTM2016-76327-C3-2-P and MTM2016-
79661-P of the Spanish Ministerio of Econom\'{\i}a and Competitividad, and Research Group FQM 199 of the Junta de Andaluc\'{\i}a (Spain),  all of them include European Union FEDER support; grant 853/2017 Plan Propio University of Granada (Spain);
 Kazakhstan Ministry of Education and Science, grant 0828/GF4.

{}
\end{document}